\documentclass[11pt,reqno,a4paper]{amsart}

\usepackage{amsmath,amsthm,verbatim,amscd,amssymb,setspace,enumitem,hyperref,exscale,color}

\renewcommand{\epsilon}{\varepsilon}

\newtheoremstyle{fancy}{}{}{\itshape}{}{\textbf\bgroup}{.\egroup}{ }{}
\newtheoremstyle{fancy2}{}{}{\rm}{}{\textbf\bgroup}{.\egroup}{ }{}

\theoremstyle{fancy}
\newtheorem{theorem}{Theorem}[section]

\newcounter{mtheorem}
\newtheorem{mtheorem}[mtheorem]{Theorem}

\newtheorem{mcor}[mtheorem]{Corollary}

\theoremstyle{fancy2}

\numberwithin{equation}{section}

\begin{document}
\title{On the Smoothability of certain K\"ahler cones}
\date{\today}
\author{Ronan J.~Conlon}
\address{D\'epartement de Math\'ematiques, Universit\'e du Qu\'ebec \`a Montr\'eal, Case Postale 8888, Succursale
Centre-ville, Montr\'eal (Qu\'ebec), H3C 3P8, Canada}
\email{rconlon@cirget.ca}
\date{\today}
\begin{abstract}
Let $D$ be a Fano manifold that may be realised as $\mathbb{P}(\mathcal{E})$ for some rank $2$ holomorphic vector bundle $\mathcal{E}\longrightarrow Z$ over a Fano manifold $Z$. Let $k\in\mathbb{N}$ divide $c_{1}(D)$. We classify those K\"ahler cones 
of dimension $\leq4$ of the form $(\frac{1}{k}K_{D})^{\times}$ that are smoothable. As a consequence, we find that any irregular Calabi-Yau cone of dimension $\leq 4$ of this form does not admit a smoothing, leaving $K_{\mathbb{P}^{2}_{(2)}}^{\times}$ as currently the only known example of a smoothable irregular Calabi-Yau cone in these dimensions.
\end{abstract}
\maketitle
\markboth{Ronan J.~Conlon}{On the smoothability of certain K\"ahler cones}

\section{Introduction}

By a \emph{smoothing} of a Calabi-Yau cone, we mean a smooth affine variety that is asymptotic to the cone at infinity. It is in particular a deformation of the cone, which, although affine, can in general be singular and does not necessarily have to be asymptotic to the cone at infinity; see \cite[Remarks 5.3 and 5.7]{Conlon}. Interest in examples of smoothings of Calabi-Yau cones stems from the fact that they are one natural starting point for the construction of asymptotically conical (AC) Calabi-Yau manifolds. Indeed, given a smoothing, one would hope to endow it with an AC Calabi-Yau metric; see for example \cite[Section 5]{Conlon}. With this in mind, we would like to identify which of the new examples of \emph{irregular} Calabi-Yau cones recently appearing in the literature are smoothable. At present, unlike the situation in the \emph{regular} case \cite{Conlon3, Tian}, there is no general theorem in the literature asserting the existence of AC Calabi-Yau metrics on a smoothing of an irregular Calabi-Yau cone. However, one example of such a smoothing admitting AC Calabi-Yau metrics is already known. Writing $\mathbb{P}^{2}_{(k)}$ to denote the blowup of $\mathbb{P}^{2}$ at $k$ points in general position and $L^{\times}$ to denote the blowdown of the zero section of a negative holomorphic line bundle $L$ over a compact K\"ahler manifold, it is shown in \cite{futaki} that the K\"ahler cone $K_{\mathbb{P}_{(2)}^{2}}^{\times}$ is an irregular Calabi-Yau cone, which by \cite{Altmann} admits a smoothing unique up to biholomorphism. AC Calabi-Yau metrics were then constructed on this smoothing in \cite{Conlon3}.

Now, the first examples of irregular Calabi-Yau cones were constructed in \cite{Sparksy1}. However, as toric examples, by looking at their toric diagram \cite[Figure 2]{Sparky}, they are seen to be rigid by \cite{Altmann}. In \cite{Cvetic2, Cvetic1}, these examples were shown to fit into a larger family of toric Calabi-Yau cones, but again, using \cite{Altmann}, one can read from their toric diagrams \cite[Figure 2]{Sparks23} that this family too comprises of rigid Calabi-Yau cones. More toric examples were then constructed in \cite{futaki}, where it is not only shown that $K_{\mathbb{P}_{(2)}^{2}}^{\times}$ is a Calabi-Yau cone, but that $K_{D}^{\times}$ admits a (potentially irregular) toric Calabi-Yau cone metric for $D$ a toric Fano manifold, not necessarily admitting a K\"ahler-Einstein metric. By \cite{Altmann}, these Calabi-Yau cones are all rigid in dimensions $n\geq4$, and the only other $3$-dimensional irregular example we find in \cite{futaki}, namely $K_{\mathbb{P}_{(1)}^{2}}^{\times}$, is also rigid; see \cite[Section 9]{Altmann}. (The Calabi-Yau cone $K_{\mathbb{P}_{(1)}^{2}}^{\times}$ was also one of the irregular examples of \cite{Sparksy1}; see \cite[Section 7]{Sparkss}.) More recently, examples of (once again, potentially irregular) non-toric Calabi-Yau cones, including those examples found in the extension \cite{Sparkles} of \cite{Sparksy1}, were given by \cite[Theorem 1.5]{mabuchi}. In this article, we address the issue of the smoothability of the examples in low dimensions resulting from this theorem.

Our main result is the following. We write $\mathbb{Q}^{n}$ to denote the smooth quadric hypersurface in $\mathbb{P}^{n+1}$.
\pagebreak

\begin{mtheorem}
Let $D$ be a Fano manifold of dimension $\leq 3$ that may be realised as $\mathbb{P}(\mathcal{E})$ for some rank $2$ holomorphic vector bundle $\mathcal{E}\longrightarrow Z$ over a Fano manifold $Z$. Let $k\in\mathbb{N}$ divide $c_{1}(D)$. Then $(\frac{1}{k}K_{D})^{\times}$ is smoothable if and only if
\begin{enumerate}
\item[{\rm (i)}] $D=\mathbb{P}^{1}$ and $k=1$. In this case, the smoothing is biholomorphic to $T^{*}\mathbb{S}^{2}$.
\item[{\rm (ii)}] $D=\mathbb{P}^{1}\times\mathbb{P}^{1}$ and $k=1$. In this case, the smoothing is biholomorphic to $\mathbb{P}^{3}\setminus\operatorname{quadric}=T^{*}\mathbb{RP}^{3}$.
\item[{\rm (iii)}] $D=\mathbb{P}^{1}\times\mathbb{P}^{1}$ and $k=2$. In this case, the smoothing is biholomorphic to $T^{*}\mathbb{S}^{3}$.
\item[{\rm (iv)}] $D=\mathbb{P}_{\mathbb{P}^{2}}(T_{\mathbb{P}^{2}})$ and $k=2$. In this case, the smoothing is biholomorphic to the
complement of a hyperplane section of the Segre embedding of $\mathbb{P}^{2}\times\mathbb{P}^{2}$ in $\mathbb{P}^{8}$.
\end{enumerate}
\end{mtheorem}

Since every instance of $D$ in Theorem A admits a K\"ahler-Einstein metric, the corresponding smoothable Calabi-Yau cone $(\frac{1}{k}K_{D})^{\times}$ must be regular, and in each case, Stenzel \cite{Stenzel} has constructed an explicit AC Calabi-Yau metric on the smoothing (see also \cite{Conlon3, Tian}). Thus, we arrive at the following corollary.
\begin{mcor}
Let $D$ and $k$ be as in Theorem A. Then every smoothable Calabi-Yau cone of dimension $\leq 4$ of the form $(\frac{1}{k}K_{D})^{\times}$ is regular.
\end{mcor}
\noindent In particular, we see that none of the irregular Calabi-Yau cones of dimension $\leq 4$ resulting from \cite[Theorem 1.5]{mabuchi} are smoothable. This leaves $K^{\times}_{\mathbb{P}^{2}_{(2)}}$ as currently the only known example of a smoothable irregular Calabi-Yau cone
in dimensions $\leq 4$.

We remark that if \cite[Conjecture 5.5.1]{Bela} holds true, then in fact the Calabi-Yau cones of Theorem \nolinebreak A would be the \emph{only} examples (in all dimensions) of smoothable Calabi-Yau cones  of the form $(\frac{1}{k}K_{D})^{\times}$ for $D$ Fano and taking the form $\mathbb{P}(\mathcal{E})$ for some rank $2$ holomorphic vector bundle $\mathcal{E}\longrightarrow Z$ over a Fano manifold $Z$, and for $k$ dividing $c_{1}(D)$. In particular, this would leave $K^{\times}_{\mathbb{P}^{2}_{(2)}}$ as currently the only known example in \emph{all} dimensions of a smoothable irregular Calabi-Yau cone.

Our proof of Theorem A follows from an observation in \cite{Conlon} that for $D$ Fano, $(\frac{1}{k}K_{D})^{\times}$ is smoothable if and only if $D$ appears as an anti-canonical divisor in a Fano manifold of one dimension greater. This result, combined with the classification theory of Fano manifolds, then yields Theorem A.

\subsection*{Acknowledgements}

The author wishes to thank Hans-Joachim Hein for comments on a preliminary version of this article.

\section{Proof of Theorem A}

We define the index $r(X)$ of a Fano manifold $X$ as $$r(X)=\max\{k\in\mathbb{N}\,|\,\textrm{$k$ divides $c_{1}(X)$ in $\operatorname{Pic}(X)$}\}.$$
The index $r(X)$ of a Fano manifold $X$ of dimension $n$ satisfies $r(X)\leq n+1$ with $r(X)=n+1$ if and only if $X=\mathbb{P}^{n}$, and $r(X)=n$ if and only if $X=\mathbb{Q}^{n}$ \cite[Corollary 3.1.15]{Fano}. If $r(X)=n-1$, then we call $X$ a \emph{del Pezzo variety} or a \emph{del Pezzo $n$-fold}.

Let $D$ be a Fano manifold of dimension $n-1$. By \cite[Proposition 5.1(i)]{Conlon}, the cone $(\frac{1}{k}K_{D})^{\times}$ is smoothable with smoothing $M$ if and only if $M =X\setminus D$ for some $n$-dimensional Fano manifold $X$ of index at least $2$ containing $D$ as an anti-canonical divisor such that $-K_{X}=(k+1)[D]$ (where, here and throughout, we write $[D]$ to denote the line bundle on $X$ induced by $D$). Thus, Theorem A will be a consequence of the following theorem.

\begin{theorem}\label{dunno}
Let $X$ be a Fano manifold of dimension $\leq 4$ and let $D$ be a smooth divisor in $X$ satisfying $-K_{X}=\alpha[D]$ for some $\alpha\in\mathbb{N},\,\alpha\geq2$. Assume that $D=\mathbb{P}(\mathcal{E})$ for some rank $2$ holomorphic vector bundle $\mathcal{E}\longrightarrow Z$ over a Fano manifold $Z$. Then the only possibilities for $\alpha$ and $(X,\,D)$ are:
\begin{enumerate}
  \item[{\rm (i)}] $\alpha=2$ and $(X,\,D)=(\mathbb{P}^{1}\times\mathbb{P}^{1},\,\mathbb{P}^{1}),\,(\mathbb{P}^{3},\,\mathbb{P}^{1}\times\mathbb{P}^{1});$ or
  \item[{\rm (ii)}] $\alpha=3$ and $(X,\,D)=(\mathbb{P}^{2},\,\mathbb{P}^{1}),\,(\mathbb{Q}^{3},\,\mathbb{P}^{1}\times\mathbb{P}^{1}),\,
(\mathbb{P}^{2}\times\mathbb{P}^{2},\,\mathbb{P}_{\mathbb{P}^{2}}(T_{\mathbb{P}^{2}})).$
\end{enumerate}
\end{theorem}

\noindent Notice that, by adjunction, the assumptions on $X$ and $D$ here imply that $D$ is necessarily Fano. By \cite[Theorem 1.6]{Szurek3}, it subsequently follows that $Z$ is also necessarily Fano.

\begin{proof}[Proof of Theorem \ref{dunno}]
First suppose that $\dim X=2$. Then $-K_{X}=\alpha[D]$, where the only possibilities for $\alpha$ are $\alpha=2,\,3$.
If $\alpha=2$, then we have that $(X,\,D)=(\mathbb{P}^{1}\times\mathbb{P}^{1},\,\mathbb{P}^{1})$. If $\alpha=3$, then we have that $(X,\,D)=(\mathbb{P}^{2},\,\mathbb{P}^{1})$.

Next suppose that $\dim X=3$. Then $-K_{X}=\alpha[D]$, where the only possibilities for $\alpha$ are $\alpha=2,\,3,\,4$. In this case, $D=\mathbb{P}^{1}\times\mathbb{P}^{1}$ or $D=\mathbb{P}_{\mathbb{P}^{1}}(\mathcal{O}_{\mathbb{P}^{1}}(1)\oplus\mathcal{O}_{\mathbb{P}^{1}})=\mathbb{P}^{2}_{(1)}$. If $\alpha=3$, then we have that $(X,\,D)=(\mathbb{Q}^{3},\,\mathbb{P}^{1}\times\mathbb{P}^{1})$. If $\alpha=4$, then necessarily $D=\mathbb{P}^{2}$, a case which clearly cannot occur.

If $\alpha=2$ and the index of $X$ is $4$, then necessarily $(X,\,D)=(\mathbb{P}^{3},\,\mathbb{P}^{1}\times\mathbb{P}^{1})$. Otherwise, $X$ has index $2$ and so is a del Pezzo variety. Since $D=\mathbb{P}^{1}\times\mathbb{P}^{1}$ or $\mathbb{P}^{2}_{(1)}$, we have that necessarily $D^{3}=(-K_{D})^{2}=8$. But from \cite[Table 12.1]{Fano}, we read that a corresponding $X$ with this property does not exist. Thus, we can rule out this case.

Now suppose that $\dim X=4$. Then $-K_{X}=\alpha[D]$, where the only possibilities for $\alpha$ are $\alpha=2,\,3,\,4,\,5$. If $\alpha=5$, then we must have $(X,\,D)=(\mathbb{P}^{4},\,\mathbb{P}^{3})$, again a case that we can rule out since $\mathbb{P}^{3}$ does not carry the desired $\mathbb{P}^{1}$-bundle structure \cite{Szurek}. If $\alpha=4$, then we must have $(X,\,D)=(\mathbb{Q}^{4},\,\mathbb{Q}^{3})$, also a case that we can rule out by \cite{Szurek}.

If $\alpha=3$, then $X$ is a del Pezzo variety with $(-K_{D})^{3}=(2([D]|_{D}))^{3}=8D^{4}$, so that $8$ divides $(-K_{D})^{3}$. Since $\dim X=4$, we see from \cite[Table 12.1]{Fano} that necessarily $D^{4}\leq6$, so that \nolinebreak $(-K_{D})^{3}\leq 48$. Also, since $-K_{D}=2([D]|_{D})$, $D$ must have index $2$ in this case. (If $D$ had index $4$, then necessarily $D=\mathbb{P}^{3}$ which cannot occur \cite{Szurek}). Therefore $D$ is a del Pezzo threefold with $(-K_{D})^{3}=8,\,16,\,24,\,32,\,40$ or $48$. We can now read from \cite{Szurek} that the only possibilities
are that $(-K_{D})^{3}=48$, so that $D^{4}=6$, and that $D=\mathbb{P}_{\mathbb{P}^{2}}(T_{\mathbb{P}^{2}})$ or
$\mathbb{P}^{1}\times\mathbb{P}^{1}\times\mathbb{P}^{1}$. Thus, $X$ is a del Pezzo variety of dimension $4$ with anti-canonical divisor $D=\mathbb{P}_{\mathbb{P}^{2}}(T_{\mathbb{P}^{2}})$ or $\mathbb{P}^{1}\times\mathbb{P}^{1}\times\mathbb{P}^{1}$, satisfying $D^{4}=6$. Reading from \cite[Table 12.1]{Fano}, we subsequently deduce that the only possibility is that $(X,\,D)=(\mathbb{P}^{2}\times\mathbb{P}^{2},\,\mathbb{P}_{\mathbb{P}^{2}}(T_{\mathbb{P}^{2}}))$.

Finally, let us consider the case $\alpha=2$. Then we have that $-K_{D}=[D]|_{D}$ by adjunction. If the index of $X$ is $4$, then we must have that $(X,\,D)=(\mathbb{Q}^{4},\,\mathbb{Q}^{3})$, a case that may be ruled out by \cite{Szurek} as above. We therefore henceforth assume that the index of $X$ is $2$.

First suppose that $X$ has Picard number $\rho(X)\geq2$. Then $X$ must appear on \cite[Table 12.7]{Fano} with $D^{4}=(-K_{D})^{3}$.
Using the fact that $\rho(X)=\rho(D)$ by the Lefschetz theorem on hyperplane sections, we read from this table that necessarily $2\leq\rho(D)\leq4$ and $D^{4}\leq 40$. From \cite{Szurek}, it then follows that the only possible values for $D^{4}$ are $30,\,36,\,38$ (any other Fano threefold $D$ on that list with $D^{4}\leq 40$ has $\rho(X)\geq5$). From \cite[Table 12.7]{Fano}, we subsequently read that the only possibility is that $D^{4}=30$, and correspondingly, that $D$ is number $23$ on \cite[Table 12.3]{Fano}. But since this Fano manifold does not appear on the list in \cite{Szurek}, we can rule out this case.

The final case to consider is when $\rho(X)=1$. Then $X$ is a ``Fano fourfold of the first species'' as defined by \cite{Wilson}, and from that paper we read that necessarily $2\leq D^{4}\leq18$ in $X$, so that $2\leq (-K_{D})^{3}\leq 18$. Since $D$ must lie on the list \cite{Szurek}, we see that necessarily $D=\mathbb{P}^{1}\times\mathbb{P}_{(k)}^{2}$, $k=6,\,7,\,8$, so that $\rho(D)>1$. Using again the fact that $\rho(X)=\rho(D)$, we derive a contradiction with the assumption that $\rho(X)=1$. Thus, this case may be ruled out as well.

This completes the proof of Theorem \ref{dunno}.
\end{proof}

\bibliographystyle{amsplain}
\bibliography{ref}

\def\cprime{$'$} \def\cprime{$'$}
\providecommand{\bysame}{\leavevmode\hbox to3em{\hrulefill}\thinspace}
\providecommand{\MR}{\relax\ifhmode\unskip\space\fi MR }
\providecommand{\MRhref}[2]{%
  \href{http://www.ams.org/mathscinet-getitem?mr=#1}{#2}
}
\providecommand{\href}[2]{#2}
\begin{thebibliography}{10}

\bibitem{Altmann}
K.~Altmann, \emph{The versal deformation of an isolated toric {G}orenstein
  singularity}, Invent. Math. \textbf{128} (1997), no.~3, 443--479. \MR{1452429
  (98g:14006)}

\bibitem{Bela}
M.~C. Beltrametti and A.~J. Sommese, \emph{The adjunction theory of complex
  projective varieties}, de Gruyter Expositions in Mathematics, vol.~16, Walter
  de Gruyter \& Co., Berlin, 1995. \MR{1318687 (96f:14004)}

\bibitem{Sparky}
S.~Benvenuti, S.~Franco, A.~Hanany, D.~Martelli, and J.~Sparks, \emph{{An
  Infinite family of superconformal quiver gauge theories with Sasaki-Einstein
  duals}}, JHEP \textbf{0506} (2005), 064.

\bibitem{Conlon}
R.~J. Conlon and H.-J. Hein, \emph{Asymptotically conical {C}alabi-{Y}au
  manifolds, {I}}, Duke Math. J. \textbf{162} (2013), 2855--2902.

\bibitem{Conlon3}
\bysame, \emph{Asymptotically conical {C}alabi-{Y}au manifolds, {II}},
  arXiv:1301.5312v2 (2014).

\bibitem{Cvetic2}
M.~Cvetic, H.~Lu, D.~N. Page, and C.N. Pope, \emph{{New Einstein-Sasaki spaces
  in five and higher dimensions}}, Phys.Rev.Lett. \textbf{95} (2005), 071101.

\bibitem{Cvetic1}
\bysame, \emph{{New Einstein-Sasaki and Einstein spaces from Kerr-de Sitter}},
  JHEP \textbf{0907} (2009), 082.

\bibitem{Sparks23}
S.~Franco, A.~Hanany, D.~Martelli, J.~Sparks, D.~Vegh, and B.~Wecht,
  \emph{Gauge theories from toric geometry and brane tilings}, J. High Energy
  Phys. (2006), no.~1, 128, 40 pp. (electronic). \MR{2201204 (2006k:81297)}

\bibitem{futaki}
A.~Futaki, H.~Ono, and G.~Wang, \emph{Transverse {K}\"ahler geometry of
  {S}asaki manifolds and toric {S}asaki-{E}instein manifolds}, J. Differential
  Geom. \textbf{83} (2009), no.~3, 585--635. \MR{2581358}

\bibitem{Sparksy1}
J.~P. Gauntlett, D.~Martelli, J.~Sparks, and D.~Waldram,
  \emph{Sasaki-{E}instein metrics on {$S^2\times S^3$}}, Adv. Theor. Math.
  Phys. \textbf{8} (2004), no.~4, 711--734. \MR{2141499 (2006m:53067)}

\bibitem{Sparkles}
\bysame, \emph{{A New infinite class of Sasaki-Einstein manifolds}},
  Adv.Theor.Math.Phys. \textbf{8} (2006), 987--1000.

\bibitem{Fano}
V.~A. Iskovskikh and Y.~G. Prokhorov, \emph{Fano varieties}, Algebraic
  geometry, {V}, Encyclopaedia Math. Sci., vol.~47, Springer, Berlin, 1999,
  pp.~1--247. \MR{1668579 (2000b:14051b)}

\bibitem{mabuchi}
T.~Mabuchi and Y.~Nakagawa, \emph{New examples of {S}asaki-{E}instein
  manifolds}, Tohoku Math. J. (2) \textbf{65} (2013), no.~2, 243--252.
  \MR{3079287}

\bibitem{Sparkss}
D.~Martelli and J.~Sparks, \emph{{Toric geometry, Sasaki-Einstein manifolds and
  a new infinite class of AdS/CFT duals}}, Commun.Math.Phys. \textbf{262}
  (2006), 51--89.

\bibitem{Stenzel}
M.~B. Stenzel, \emph{Ricci-flat metrics on the complexification of a compact
  rank one symmetric space}, Manuscripta Math. \textbf{80} (1993), no.~2,
  151--163. \MR{1233478 (94f:32020)}

\bibitem{Szurek}
M.~Szurek and J.~A. Wi{\'s}niewski, \emph{Fano bundles of rank {$2$} on
  surfaces}, Compositio Math. \textbf{76} (1990), no.~1-2, 295--305, Algebraic
  geometry (Berlin, 1988). \MR{1078868 (92e:14037)}

\bibitem{Szurek3}
\bysame, \emph{Fano bundles over {${\bf P}^3$} and {$Q_3$}}, Pacific J. Math.
  \textbf{141} (1990), no.~1, 197--208. \MR{1028270 (91g:14036)}

\bibitem{Tian}
G.~Tian and S.-T. Yau, \emph{Complete {K}\"ahler manifolds with zero {R}icci
  curvature. {II}}, Invent. Math. \textbf{106} (1991), no.~1, 27--60.
  \MR{1123371 (92j:32028)}

\bibitem{Wilson}
P.~M.~H. Wilson, \emph{Fano fourfolds of index greater than one}, J. Reine
  Angew. Math. \textbf{379} (1987), 172--181. \MR{903639 (89b:14058)}

\end{thebibliography}
\end{document}